\documentclass[12pt]{amsart}
\usepackage[utf8]{inputenc}
\usepackage{graphicx}
\usepackage{epstopdf}
\usepackage{inputenc}
\usepackage{enumitem}
\usepackage{amsmath}
\usepackage{amssymb}
\usepackage{mathtools}
\usepackage[usenames,dvipsnames,x11names,svgnames]{xcolor}
\usepackage[colorlinks=true,linkcolor=NavyBlue,citecolor=DarkGreen, urlcolor=blue]{hyperref}
\usepackage{dsfont}

\newtheorem{theorem}{Theorem}
\newtheorem*{theorem*}{Theorem}

\newtheorem{lemma}{Lemma}

\newtheorem*{acknowledgements*}{Acknowledgements}

\newtheorem*{dichotomy}{Dichotomy}

\newcommand{\mb}{\mathbb}
\newcommand{\mc}{\mathcal}
\newcommand{\mf}{\mathfrak}

\newcommand{\leqs}{\leqslant }
\newcommand{\geqs}{\geqslant }

\makeatletter
\def\blfootnote{\gdef\@thefnmark{}\@footnotetext}
\makeatother

\def\house#1{\setbox1=\hbox{$\,#1\,$}%
\dimen1=\ht1 \advance\dimen1 by 2pt \dimen2=\dp1 \advance\dimen2 by 2pt
\setbox1=\hbox{\vrule height\dimen1 depth\dimen2\box1\vrule}%
\setbox1=\vbox{\hrule\box1}%
\advance\dimen1 by .4pt \ht1=\dimen1
\advance\dimen2 by .4pt \dp1=\dimen2 \box1\relax}

\begin{document}
	\title[A dichotomy for extreme values of zeta and Dirichlet $L$-functions]{A dichotomy for extreme values of \\ zeta and Dirichlet $L$-functions}

\author[A. Bondarenko]{Andriy Bondarenko}
\address{Department of Mathematical Sciences, Norwegian University of Science and Technology (NTNU), 7491 Trondheim, Norway} 
\email{andriybond@gmail.com}

\author[P. Darbar]{Pranendu Darbar}
\address{Department of Mathematical Sciences, Norwegian University of Science and Technology (NTNU), 7491 Trondheim, Norway} 
\email{darbarpranendu100@gmail.com}

\author[M.~V.~Hagen]{Markus V. Hagen}
\address{Department of Mathematical Sciences, Norwegian University of Science and Technology (NTNU), 7491 Trondheim, Norway} 
\email{markus.v.hagen@ntnu.no}

\author[W. Heap]{Winston Heap}
\address{Department of Mathematics, Shandong University, Jinan, Shandong 250100, China}
\email{winstonheap@gmail.com}

\author[K. Seip]{Kristian Seip}
\address{Department of Mathematical Sciences, Norwegian University of Science and Technology (NTNU), 7491 Trondheim, Norway} 
\email{kristian.seip@ntnu.no}

\thanks{Research supported in part by Grant 275113 of the Research Council of Norway.
The work of Darbar is 
funded by that grant through the Alain
Bensoussan Fellowship Programme of the European Research Consortium for Informatics and Mathematics.}

	\maketitle
	
	\begin{abstract}
	We exhibit large values of the Dedekind zeta function of a cyclotomic field on the critical line. This implies a dichotomy whereby one either has improved lower bounds for the maximum of  the Riemann zeta function, or large values of Dirichlet $L$-functions on the level of the Bondarenko--Seip bound.
	\end{abstract}
		
	\section{introduction}
	
	Extreme values play a key role in the value distribution theory of $L$-functions. In \cite{Sound res}, Soundararajan introduced a versatile method for obtaining large values of $L$-functions in a variety of families. Applied to the Riemann zeta function, this gave 
	\[
	\max_{t\in [T,2T]}|\zeta(\tfrac12+it)|\geqs \exp\bigg((1+o(1))\sqrt{\frac{\log T}{\log\log T}}\bigg),
	\]
	improving on previous results of Balasubramanian--Ramachandra \cite{BR}.  By combining a modification of Soundararajan's version of the resonance method along with the burgeoning connections with GCD sums and insights of Aistleitner \cite{A}, the first and fifth authors showed \cite{BS1} that 
	\[
	\max_{t\in[0,T]}|\zeta(\tfrac12+it)|
	\geqs 
	\exp\bigg(\big(\frac{1}{\sqrt{2}}+o(1)\big)\sqrt{\frac{\log T\log\log\log T}{\log\log T}}\bigg).
	\]
	The constant $1/\sqrt{2}$ was subsequently improved to $1$ in \cite{BS2} with the current best being $\sqrt{2}$ due to de la Bret\`eche--Tenenbaum \cite{dlBT}. 
	
	These techniques have since been applied in a variety of settings \cite{AMM, C, CM1, CM2, Y}. A severe constraint, however, is that the $L$-functions under consideration must have positive coefficients. This excludes many $L$-functions of interest and so, for instance, it is not known whether Dirichlet $L$-functions exhibit such large values. Currently, Soundararajan's resonance method gives the best known bounds 
	\[ \max_{t\in[T,2T]}|L(1/2+it,\chi)|\geqs  \exp\bigg((1+o(1))\sqrt{\frac{\log T}{\log\log T}}\bigg)\] for non-principal Dirichlet characters $\chi$.  
	
	A plentiful source of $L$-functions with positive coefficients are provided by the Dedekind zeta functions. In this paper, we consider the Dedekind zeta function $	\zeta_\mb{K}(s)$ attached to a cyclotomic field $\mb{K}=\mb{Q}(\omega_q)$ with $\omega_q $ a $q$th root of unity for $q>2$. 
	Of special interest is the factorisation 
	\begin{equation}\label{ded factor}
	\zeta_\mb{K}(s)=\zeta(s)\prod_{\chi\neq \chi_0\!\!\!\!\pmod{q}}L(s, \chi'),
	\end{equation}		
	\vspace{-0.1cm}where the product is over all non-principal Dirichlet characters $\chi$ modulo $q$ and $\chi'$ is the character which induces $\chi$ if $\chi$ is not primitive and $\chi^\prime=\chi$ otherwise. We will establish the following bound which thus yields an assertion about the \emph{interplay} between extreme values of the functions $L(s,\chi)$. 
	
	\begin{theorem}\label{main theorem}
	Let  $\mb{K}=\mb{Q}(\omega_q)$ and $A$ be an arbitrary positive number. If $T$ is sufficiently large, then uniformly for $q\ll (\log_2 T)^A$, 
	\begin{equation} \label{eq:mainest}
	\max_{t\in[0,T]}|\zeta_\mb{K}(\tfrac12+it)|\geqs \exp\left(\bigg(1+o(1)\bigg)\sqrt{\phi(q)}\sqrt{\frac{\log T \log \log \log T}{\log \log T}}\right).
	\end{equation}
	\end{theorem}

Thanks to the extra factor $\sqrt{\phi(q)}$ in the exponent on the right-hand side of \eqref{eq:mainest}, we obtain the following consequence of Theorem \ref{main theorem} and the factorisation \eqref{ded factor}. 

\begin{dichotomy}Either there exist $L(s,\chi)$ for non-principal Dirichlet characters $\chi$ satisfying 
\[
	\max_{t\in[0,T]}|L(\tfrac12+it,\chi)|
	\geqs 
	\exp\bigg(c\sqrt{\frac{\log T\log\log\log T}{\log\log T}}\bigg)
\] 
for some sufficiently small $c$, or we can exhibit even larger values of $|\zeta(\tfrac12+it)|$.
\end{dichotomy}


As an extreme possible event, if  $L(\tfrac{1}{2}+it,\chi)$ is less than  $\exp(c\sqrt{\log T/\log\log T})$ on $[0,T]$ for all non-principal Dirichlet characters $\chi$ modulo $q$ and some prime $q\sim \frac{1}{4c^2}\log\log\log T$, then we will have  
\[ \max_{t\in [0,T]} |\zeta(\tfrac{1}{2}+it)| \geqs  \exp\bigg(\big(\frac{1}{4c}+o(1)\big) \sqrt{\frac{\log T}{\log\log  T}}\times \log\log\log T\bigg).\]
Perhaps more likely is that Dirichlet $L$-functions actually \emph{do} obtain large values with an extra power $\sqrt{\log\log\log T}$, and that Theorem~\ref{main theorem} is picking out  simultaneous occurrences of this.   

Previously, large values of the Dedekind zeta function of a general field $\mb{K}$ were given by Li in \cite[Thm. 1.1.1]{Li thesis}. Her results guaranteed the existence of arbitrarily large $t$ such that
	\begin{equation} \label{eq:li}
	|\zeta_\mb{K}(\tfrac12+it)|\geqs \exp\left(c\sqrt{\frac{\log t \log \log \log t}{\log \log t}}\right),
	\end{equation}
where one could take any $c<1/\sqrt{2}$ for Galois extensions and any $c<1/(\sqrt{2}[\mb{K}:\mb{Q}])$ in general. However, with these bounds one cannot deduce our dichotomy. 

Our improvement of the constant in \eqref{eq:li} is afforded by the following observation. Note that by \eqref{ded factor} we may write $\zeta_{\mb{Q}(\omega_q)}(s)=\sum_{n\geqs 1} a(n)n^{-s}$ where for primes $p\nmid q$ we have 	
	\[
a(p)=\left\{
\begin{array}
[c]{ll}
\phi(q) & \text{if}\; \, p\equiv 1 \pmod q,\\
0 & \text{otherwise}. 
\end{array}
\right.
\]
For the purposes of this discussion we may ignore primes $p|q$ since they are only finite in number. On applying the resonance method we obtain a lower bound for the maximum which is roughly of the form 
\[
\prod_{p\equiv 1(\!\!\!\!\!\!\mod q)}\bigg(1+\frac{\phi(q)r(p)}{p^{1/2}}\bigg),
\]	
where $r$ are the resonator coefficients. The fact that $$\phi(q)\sum_{p\equiv 1(\!\!\!\!\!\!\mod q),\,p\leqs x}1\sim \sum_{p\leqs x}1$$ means that this lower bound is essentially $\prod_p (1+r(p)/p^{1/2})$, i.e.\,\,we are in the same situation as for the Riemann zeta function and nothing seems to have been gained. However, the fact that our resonator only needs to be supported on a smaller set of primes, $p\equiv 1 (\mod q)$, allows us to take it larger whilst still matching the other constraints of the argument. Precisely, we can take it larger by  a factor of $\sqrt{\phi(q)}$. In order to balance this, we need to take larger primes than usual.

It is likely that the methods of de la Bret\`eche--Tenenbaum \cite{dlBT} can improve the exponent on the right hand side of \eqref{eq:mainest} by a factor of $\sqrt{2}$. However, this would not affect our dichotomy in any essential way and so in the interests of keeping our exposition simpler we have not pursued this line of inquiry.

We close this introduction by mentioning the possibility of extending our result to other Dedekind zeta functions. To this end, let $L/\mathbb{Q}$ be a Galois extension. Then the $p$th coefficient of $\zeta_L(s)$ is $[L:\mathbb{Q}]$ if $p$ splits completely in $L$. The density of the primes that split completely in $L$ is $1/[L:\mathbb{Q}]$ by Chebotarev's density theorem. This should be thought of as the analogue of $1/\phi(q)$ in the Siegel--Walfisz theorem in our case. It seems plausible then that one should be able to establish a counterpart to our main theorem, with constant $\sqrt{[L:\mathbb{Q}]}$ instead of $\sqrt{\phi(q)}$ by modifying the resonator coefficients slightly. When $L/\mathbb{Q}$ is Galois, $\zeta_L(s)$ factors into a product of Artin $L$-function according to the decomposition of the regular representation of $\text{Gal}(L/\mathbb{Q})$ into irreducible representations. We could then get a similar dichotomy as in the cyclotomic case. In particular, when $\text{Gal}(L/\mathbb{Q})$ is abelian, Dirichlet $L$-functions would be replaced by Hecke $L$-functions. For the non-abelian case this is more subtle in general. 

This paper contains two additional sections. We prepare for the proof of Theorem~\ref{main theorem} in the next section by developing the required novel extremal GCD-type sums. The actual proof of Theorem~\ref{main theorem} is carried out in Section~\ref{proof:thm1}. 

We will in what follows use the notations $\log_2 x\coloneqq \log\log x$ and
$\log_3 x\coloneqq \log\log\log x$, and we will use the convention that
\[ \widehat{f}(t)\coloneqq \int_{\mathbb R} f(x) e^{-itx} dx \]
for $f$ an integrable function on $\mathbb R$.
	
	\section{Extremal GCD-type sums}
	
	We will now construct the extremal versions of the sums that appear in the resonance method. We follow the scheme of  \cite[Sec. 2]{BS1} closely, the main differences being that we need to account for the coefficients of the Dedekind zeta function and that we are picking primes that are congruent to $1$ modulo $q$.
	
	Let $0<\gamma<1$ be a parameter to be chosen later and $P_q$ the set of all primes $p$ such that 
	\[
	p\equiv 1 \pmod q \quad \text{ and } \quad e\phi(q)\log N \log_2 N<p\leqs \phi(q)\log N \exp\left((\log_2 N)^\gamma\right) \log_2 N.
	\] 
	Here $N$ is a large integer to be chosen later as $N=\lfloor T^\eta\rfloor$. 
For any $c_q\geqs 1$, we define the multiplicative function $f(n)$ to be supported on the set of square-free numbers such that
	\[
	f(p)=\bigg\{
	\begin{array}
	[c]{ll}
	c_q \sqrt{\frac{\log N \log_2 N}{\log_3 N}} \frac{1}{\sqrt{p}(\log p -\log_2 N-\log_3 N-\log\phi(q))}& \text{if}\; \, p\in P_q,\\
	0 & \text{otherwise}.
	\end{array}
	\]
	Eventually we will take $c_q=\sqrt{\phi(q)}$, but we keep it general for the time being.
	Let $P_{k, q}$ be the set of all primes $p$ such that 
	\[ p\equiv 1 \pmod q \quad \text{and} \quad \phi(q) e^k\log N \log_2 N<p\leqs \phi(q) e^{k+1} \log N \log_2 N\] for $k=1, \ldots, \lfloor(\log_2 N)^{\gamma}\rfloor$. 
	Fix $a$ satisfying $1<a<1/\gamma$. Let ${M}_{k, q}$ be the set of those integers having at least $\frac{a \log N}{k^2  \log_3 N}$ prime divisors in $P_{k, q}$. Also let ${M}_{k, q}'$ be the set of integers from ${M}_{k, q}$ that have prime divisors only in $P_{k, q}$. Set
	\begin{align*}
	\mc{M}_q\coloneqq  \text{supp}(f)\setminus \cup_{k=1}^{[(\log_2 N)^{\gamma}]}M_{k, q}. 
	\end{align*}
	\begin{lemma}
		We have $|\mathcal{M}_q|\leqs N$ uniformly for $q\ll (\log_2 N)^A$. 
	\end{lemma}
	\begin{proof}Note that
	\[
	|\mc{M}_q|\leqs \prod_{k=1}^{\lfloor(\log_2 N)^\gamma\rfloor} \sum_{j=1}^{\lfloor \frac{a \log N}{k^2  \log_3 N}\rfloor}
	\binom{|P_{k,q}|}{j}.
	\]
	By the Siegel--Walfisz theorem, we have 
	\[
	|P_{k,q}|
	 \leqs 
	 (1+o(1))e^{k+1}\log N
	\] 
	provided $q\leqs \log^A(\phi(q)e^{k+1}\log N\log_2 N)$ which is satisfied for $q\ll (\log_2 N)^A$. The remainder of the proof now follows directly that of Lemma 2 of \cite{BS1}. 
	\end{proof}
	
	For $\mb{K}=\mb{Q}(\omega_q)$ write 
	\[
	\zeta_\mb{K}(s)
	=
	\zeta(s)\prod_{\chi\neq \chi_0 \!\!\!\!\pmod{q}}L(s, \chi')
	=
	\sum_{n=1}^\infty \frac{a(n)}{n^s}
	\] 
	so that, as mentioned above, if $p\nmid q$ then 
	\[
	a(p)=\left\{
	\begin{array}
	[c]{ll}
	\phi(q) & \text{if}\; \, p\equiv 1 \pmod q,\\
	0 & \text{otherwise}.
	\end{array}
	\right.
	\]
	Since our resonator only interacts with $p\equiv 1\mod q$ we need not compute $a(p)$ on the ramified primes $p|q$. We now consider the quantity 
	\begin{align*}
	A_{N, q}&\coloneqq \frac{1}{\sum_{n\in \mathbb{N}}f(n)^2}\sum_{n\in \mathbb{N}} \frac{f(n)}{n^{1/2}} \sum_{\substack{d\mid n}}a(n/d)f(d)\sqrt{d}\\
	&=\prod_{p\in P_q}\frac{\big(1+f(p)p^{-1/2}(a(p)+f(p)p^{1/2})\big)}{1+f(p)^2}\\
	&= \prod_{p\in P_q}\frac{1+f(p)^2+\frac{\phi(q)f(p)}{p^{1/2}}}{1+f(p)^2}.
	\end{align*}

	\begin{lemma}\label{A lem}
	Suppose that $c_q\leqs \sqrt{\phi(q)}$ and $q\ll (\log_2 N)^A$. Then 
		\[
		A_{N, q}=\exp\left((\gamma c_q+o(1))\sqrt{\frac{\log N \log_3 N}{\log_2 N}}\right).
		\]
		\end{lemma}
	\begin{proof}
		Note that if $c_q\leqs \sqrt{\phi(q)}$, then $f(p)=o(1)$ for all $p\in P_q$. Thus
		\begin{align*}
		A_{N,q}=\exp\bigg((1+o(1))\sum_{\substack{p\in P_q}}\frac{\phi(q)f(p)}{p^{1/2}}\bigg).
		\end{align*}
		By the Siegel--Walfisz theorem along with partial summation, we have
		\begin{align*}
		\phi(q)\sum_{\substack{p\in P_q}}\frac{ f(p)}{p^{1/2}}
		\sim &
		\,\,c_q\int_{e\phi(q)\log N\log_2 N}^{\phi(q)\log N\exp((\log_2 N)^\gamma)\log_2 N}
		\!\!\!\!\!\!\frac{\sqrt{{\log N\log_2 N}/{\log_3 N}} }{x\log x(\log x-\log_2 N-\log_3 N-\log\phi(q))}dx
		\\
		= &
		(c_q \gamma+o(1))\sqrt{\frac{\log N \log_3 N}{\log_2 N}},
		\end{align*}
		again, provided $q\ll (\log_2 N)^A$.
		\end{proof}
	
	\begin{lemma}\label{restrict lem}
	Suppose that $c_q\leqs \sqrt{{\phi(q)}}$ and $q\ll (\log_2 N)^A$. Then 
	\begin{align*}
	\frac{1}{\sum_{n\in \mathbb{N}}f(n)^2}\sum_{\substack{n\in \mathbb{N}\\ n\notin \mc{M}_{q}}} \frac{f(n)}{n^{1/2}} \sum_{\substack{d\mid n}}a(n/d)f(d)\sqrt{d}=o(A_{N, q}), \quad N\to \infty.
	\end{align*}
	\end{lemma}
	
	\begin{proof}
		We begin with 
		\begin{align*}
	&	\frac{1}{A_{N, q} \sum_{n\in \mathbb{N}}f(n)^2}\sum_{\substack{n\in \mathbb{N}\\ n\notin \mathcal{M}_{q}}} \frac{f(n)}{n^{1/2}} \sum_{\substack{d\mid n}}a(n/d)f(d)\sqrt{d}\\
	&	\leqs
		\frac{1}{A_{N, q} \sum_{n\in \mathbb{N}}f(n)^2}\sum_{k=1}^{[(\log_2 N)^{\gamma}]}\sum_{\substack{n\in M_{k, q}}} \frac{f(n)}{n^{1/2}} \sum_{\substack{d\mid n}}a(n/d)f(d)\sqrt{d}\\
		& =\sum_{k=1}^{[(\log_2 N)^{\gamma}]}\frac{1}{\prod_{p\in P_{k, q}}\left(1+f(p)^2+f(p)p^{-1/2}\right)}\sum_{n\in M_{k, q}'}\frac{f(n)}{n^{1/2}}\sum_{d\mid n}a(n/d)f(d)\sqrt{d}\\
		& \leqs \sum_{k=1}^{[(\log_2 N)^{\gamma}]}\frac{1}{\prod_{p\in P_{k, q}}\left(1+f(p)^2\right)}\sum_{n\in M_{k, q}'}f(n)^2\prod_{p\in P_{k, q}}\left(1+\frac{\phi(q)}{f(p)p^{1/2}}\right)=\sum_{k=1}^{[(\log_2 N)^{\gamma}]}E_{k, q},
		\end{align*}
		say.
		For each $k=1, \ldots, \lfloor(\log_2 N)^{\gamma}\rfloor$, 
		\begin{align*}
		E_{k, q}\leqs \frac{1}{\prod_{p\in P_{k, q}}\left(1+f(p)^2\right)}\sum_{n\in M_{k, q}'} f(n)^2\prod_{\substack{p\in P_{k, q}}}\left(1+\frac{\phi(q)}{f(p)\sqrt{p}}\right).
		\end{align*}
		By the Siegel--Walfisz theorem, for $q\ll (\log_2 N)^A$, we have
		\begin{align*}
		\prod_{\substack{p\in P_{k, q}}}&\left(1+\frac{\phi(q)}{f(p)\sqrt{p}}\right)
		\\
		&=
		\prod_{p \in P_{k,q}}		
		\left(1+\frac{\phi(q)}{c_q}(\log p-\log_2 N-\log_3 N-\log\phi(q))\sqrt{\frac{\log_3 N}{\log N \log_2 N}}\right)
		\\
		&\leqs
		\exp\left(\frac{\phi(q)}{c_q} (k+1)e^{k+1}\sqrt{\frac{\log N \log_3 N}{\log_2 N}}\right)=\exp\left(o\left(\frac{\log N}{\log_3 N}\right)\frac{1}{k^2}\right)
		\end{align*} 
		with the latter bound following since $k\leqs (\log_2 N)^{\gamma}$ and $\phi(q)\ll (\log_2 N)^A$.  
		
		Since every number in $M_{k, q}'$ has at least $\frac{a\log N}{k^2 \log_3 N}$
		prime divisors and  $f(n)$ is a multiplicative function, for any $b>1$, we have
		\[
		\sum_{n\in M_{k, q}'} f(n)^2\leqs b^{-a\frac{\log N}{k^2  \log_3 N}}\prod_{\substack{p\in P_{k, q}}}\left(1+bf(p)^2\right).
		\]
		Hence
		\[
		\frac{1}{\prod_{p\in P_{k, q}}\left(1+f(p)^2\right)}\sum_{n\in M_{k, q}'}f(n)^2\leqs b^{-a\frac{\log N}{k^2 \log_3 N}}\exp\left(\sum_{\substack{p\in P_{k, q}}}(b-1)f(p)^2\right).
		\]
		Observe that by the Siegel--Walfisz theorem,
		\begin{align*}
		\sum_{\substack{p\in P_{k, q}}}f(p)^2
		= &
		c_q^2\frac{\log N\log_2 N}{\log_3 N}\sum_{p\in P_{k,q}}\frac{1}{p(\log p-\log_2 N-\log_3 N-\log\phi(q))^2}
		\\
		\leqs &(1+o(1)) \frac{c_q^2}{\phi(q)}\frac{\log N\log_2 N}{\log_3 N}\int_{e^k\phi(q)\log N\log_2 N}^{e^{k+1}\phi(q)\log N\log_2 N}
		\!\!\!\!\!\!\frac{1}{k^2 x\log x}dx
		\\
		\leqs &
		\frac{c_q^2}{\phi(q)}(1+o(1))\frac{\log N}{k^2 \log_3 N}.
		\end{align*}
		Combining all these estimates, we find that
		\[
		E_{k, q}\ll \exp\left(\left(\frac{c_q^2}{\phi(q)} (b-1)-a\log b+o(1)\right)\frac{\log N}{k^2 \log_3 N}\right).
		\]
		Since $c_q\leqs \sqrt{\phi(q)}$ and $a>1$, on taking $b$ sufficiently close to 1 the exponent is negative giving the result.
		\end{proof}

\section{Proof of Theorem \ref{main theorem}} \label{proof:thm1}
	We follow the setup from \cite[Sec. 3]{BS1}. Let $\mathfrak{J}_q$ be the set of integers $j$ such that 
	\[
	\left[(1+T^{-1})^j, (1+T^{-1})^{j+1}\right]\cap \mc{M}_q\neq \emptyset.
	\]
	Also let $m_j$ be the minimum of $\left[(1+T^{-1})^j, (1+T^{-1})^{j+1}\right]\cap \mc{M}_q$ for all $j$ in $\mathfrak{J}_q$. Set 
	\[
	\mc{M}_q'\coloneqq \left\{m_j\, :\, j\in \mathfrak{J}_q \right\},
	\]
	and for every $m_j\in \mc{M}_q'$ let
	\[
	r(m_j)\coloneqq\bigg(\sum_{\substack{n\in \mc{M}_q\\ (1-T^{-1})^{j-1}\leqs n\leqs (1+T^{-1})^{j+2}}}f(n)^2\bigg)^{1/2}.
	\]
	We take our resonator to be 
	\[
	R(t)\coloneqq\sum_{m\in\mc{M}_q'}r(m)m^{-it}.
	\]
	Set $N=[T^{\eta}]$ for some $0<\eta\leqs 1$, and $\Phi(t)\coloneqq e^{-t^2/2}$ so that 
	\[
	\widehat{\Phi}(t)
	\coloneqq \int_\mb{R} \Phi(x)e^{-itx}dx=\sqrt{2\pi}\Phi(t).
	\] 
	\begin{lemma}\label{convolution formula}
		Suppose that $1/2\leqs \sigma<1$ and let $K(x+iy)$ be an analytic function in the horizontal strip $\sigma-2\leqs y \leqs 0$ satisfying
		\[
		\max_{\sigma-2\leqs y \leqs 0}|K(x+iy)|=O\left(1/|x|^2\right), \, |x|\to \infty.
		\] 
		Then for all real $t$, we have
		\[
		\int_{-\infty}^{\infty}\zeta_\mb{K}(\sigma+i(t+u))K(u)du=\sum_{n=1}^{\infty}\frac{\widehat{K}(\log n) a(n)}{n^{\sigma+it}}+2\pi \mathrm{Res}_{s=1}\zeta_\mb{K}(s)K(-t-i(1-\sigma)).
		\]
		\end{lemma}
	\begin{proof}
	This follows exactly as in Lemma 1 of \cite{BS2}.
	\end{proof}
	
	In \cite{BS2}, where the Riemann zeta function was considered, this was applied with $K(u)={\sin^2(u\epsilon \log T)}/{u^2\epsilon \log T}$ which had the sufficient decay properties for large $u$. In our case, since the Dedekind zeta function can potentially be much larger, we take 
	\begin{equation}\label{K weight}
	K_\mf{n}(u)\coloneqq\frac{\sin^{2\mf{n}}((\tfrac{1}{\mf{n}}\epsilon \log T)u)}{(\tfrac{1}{\mf{n}}\epsilon \log T)^{2\mf{n}-1}u^{2\mf{n}}},
	\end{equation} 
	with $\epsilon$ small and large $\mf{n}\in\mb{N}$ to be chosen. We summarise some properties of the Fourier transform in the following lemma. 
	
	\begin{lemma}\label{K trans lem}Let $K_\mf{n}(u)$ be as above. Then 
	$\widehat{K_\mf{n}}(v)$ is a real, even function supported on $|v|\leqs 2\epsilon \log T$ satisfying 
	$ 0\leqs \widehat{K_\mf{n}}(v)\leqs \widehat{K_\mf{n}}(0) $ and being decreasing on $[0,\infty)$ with
	\begin{equation}
	\label{K der bound}
	\big|\frac{d}{dv}\widehat{K_\mf{n}}(v)\big|  \leqs \frac{\widehat{K_{\mf{n}-1}}(0)}{\tfrac{1}{\mf{n}}\epsilon \log T}.
	\end{equation}
	Furthermore, for large $\mf{n}$ we have 
	\begin{equation}\label{K0 asymp}
	\widehat{K_\mf{n}}(0)\sim \sqrt{\frac{3\pi}{\mf{n}}}.
	\end{equation}
	\end{lemma}
	\begin{proof}
	The support condition, the nonnegativity, and the monotonicity follow from the convolution theorem and the fact that the Fourier transform of $\sin x/x$ is $\tfrac{1}{\pi}\chi_{[-1,1]}(t)$. 	
	For the bounds on the derivative, we have 
	\[
	\big|\frac{d}{dv}\widehat{K_\mf{n}}(v)\big|\leqs \int_\mb{R} \frac{|\sin^{2\mf{n}}((\tfrac{1}{\mf{n}}\epsilon \log T)u)|}{(\tfrac{1}{\mf{n}}\epsilon \log T)^{2\mf{n}-1} u^{2\mf{n}-1}}du\leqs \frac{1}		
	{\tfrac1n\epsilon\log T}\int_{\mb{R}} \frac{|\sin^{2\mf{n}}(u)|}{u^{2\mf{n}-1}}du\leqs \frac{\widehat{K_{\mf{n}-1}}(0)}{\tfrac{1}{\mf{n}}\epsilon \log T}.
	\]
	Finally, for large $\mf{n}$ and $c>3$ we have 
	\begin{align*}
	\widehat{K_\mf{n}}(0)
	= &
	\int_{\mb{R}}\frac{\sin^{2\mf{n}}(x)}{x^{2\mf{n}}}dx
	=
	\int_{|x|\leqs c\sqrt{{\log \mf{n}}/{{\mf{n}}}}}\frac{\sin^{2\mf{n}}(x)}{x^{2\mf{n}}}dx+O(\mf{n}^{-c/3})
	\end{align*} 
	since $\sin x/x$ is decreasing over the interval $[0,\pi]$, $\sin^{2\mf{n}}(c\sqrt{\log \mf{n}/\mf{n}})/(c\sqrt{\log \mf{n}/\mf{n}})^{2\mf{n}}\ll \mf{n}^{-c/3}$ by Taylor 		
	expansions, and the integral over $|x| \geqs \pi$ is $\ll \pi^{-2\mf{n}}$. On applying Taylor expansions again, we find that the above integral is
	\begin{align*}
	& \int_{|x|\leqs c\sqrt{{\log \mf{n}}/{{\mf{n}}}}}(1-\tfrac16x^2+O(\tfrac{(\log \mf{n})^2}{\mf{n}^2}))^{2\mf{n}}dx 
	\sim \frac{1}{\sqrt{\mf{n}}}\int_\mb{R} e^{-x^2/3}dx
	=
	\sqrt{\frac{{3\pi}}{{\mf{n}}}}.
	\end{align*}

	\end{proof}

	\begin{lemma}\label{main lemma}
		For large $T$, $c_q\leqs \sqrt{\phi(q)}$, $\mf{n}\ll (\log N)^{1/2-\delta}$ with small $\delta>0$, and $q\ll (\log_2 T)^A$ we have 
		\[
		\int_{-\infty}^{\infty}|R(t)|^2\Phi(t/T)dt\ll T\sum_{n\in \mb{N}}f(n)^2
		\]
		and 
		\begin{align*}
		&\int_{-\infty}^{\infty}\left(\sum_{n=1}^{\infty}\frac{\widehat{K_\mf{n}}(\log n)a(n)}{n^{1/2+it}}\right)|R(t)|^2\Phi(t/T)dt\\
	& \gg  \widehat{K_\mf{n}}(0)T\exp\left(c_q \gamma(1+o(1)) \sqrt{\frac{\log N \log_3 N}{\log_2 N}}\right)\sum_{n\in \mb{N}}f(n)^2.
		\end{align*}
	
		\end{lemma}
	\begin{proof}
	The first part follows similarly to Lemma 5 of \cite{BS2}. For the second part we have
		\begin{align*}
		&\int_{-\infty}^{\infty}\left(\sum_{n=1}^{\infty}\frac{\widehat{K_\mf{n}}(\log n)a(n)}{n^{1/2+it}}\right)|R(t)|^2\Phi(t/T)dt\\
		&=\sqrt{2\pi} T \sum_{m,n \in \mc{M}_q'}\sum_{k=1}^{\infty}\frac{\widehat{K_\mf{n}}(\log k) a(k)r(m)r(n)}{k^{1/2}}\Phi\left(T\log \frac{km}{n}\right).
		\end{align*}
		We wish to lower bound this using positivity. From the properties of $\widehat{K_n}(v)$ given in Lemma \ref{K trans lem}, in particular the derivative bound \eqref{K 		der bound}, we see that if $c_\mf{n}\coloneqq \widehat{K_{\mf{n}}}(0)/\widehat{K_{\mf{n}-1}}(0)$, then $\widehat{K_{\mf{n}}}(\tfrac12c_\mf{n}\cdot\tfrac{1}{\mf{n}}\epsilon 		\log T)\geqs \tfrac12\widehat{K_\mf{n}}(0)$. By \eqref{K0 asymp} we have $c_\mf{n}\sim 1$ and hence on restricting the sum to $\log k\leqs \tfrac{1}{3\mf{n}}			\epsilon \log T$, say, the above is 
		\begin{align*}
		&\gg \widehat{K_\mf{n}}(0)T\sum_{m, n\in \mc{M}_q'}\sum_{k\leqs T^{\epsilon/3\mf{n}}}\frac{a(k)r(m)r(n)}{k^{1/2}}\Phi\left(T\log \frac{km}{n}\right)
		\end{align*}
		by positivity. 
		
		 We now lower bound by summing over those integers $m,n\in\mc{M}_q'$ such that $|km/n-1|\leqs 3/T$. 
		For such terms we have $\Phi(T\log(km/n))\gg 1$ and also, similarly to equation (21) of \cite{BS2}, we find  
		\[
		\sum_{m,n\in\mc{M}_q', \,\,|km/n-1|\leqs 3/T}r(m)r(n)\geqs \sum_{m,n\in\mc{M}_q', km=n}f(m)f(n).
		\]
		This gives the lower bound
			\begin{align*}
		 \gg T\sum_{\substack{m, n\in \mc{M}_q, k\leqs T^{\epsilon/3\mf{n}}\\km=n}}\frac{a(k)f(m)f(n)}{k^{1/2}}	
		=
		T\sum_{n\in \mc{M}_q}\frac{f(n)}{n^{1/2}}\sum_{d|n, d\geqs n/T^{\epsilon/3\mf{n}}}a(n/d)f(d)d^{1/2}.
		\end{align*}
		
		It remains to remove the restriction on the divisor $d$ since then the result will follow from Lemmas \ref{A lem} and \ref{restrict lem}.
		Following Lemma 3 of \cite{BS1}, we note that since $f(n)=f((n/d)d)=f(n/d)f(d)$, the tail of this last sum is given by 
		\[
		\sum_{n\in \mc{M}_q}{f(n)^2}\sum_{d|n, d\geqs T^{\epsilon/3\mf{n}}}\frac{a(d)}{f(d)d^{1/2}}
		\leqs 
		T^{-\epsilon\delta/12\mf{n}}\sum_{n\in \mc{M}_q}{f(n)^2}\prod_{p|n}\bigg(1+\frac{a(p)}{f(p)p^{1/2-\delta}}\bigg).
		\]
		Since 
		\[
		\frac{\phi(q)}{f(p)p^{1/2-\delta}}
		\leqs 
		\frac{\phi(q)(\log_2 N)^\gamma p^{\delta}}{c_q\sqrt{\log N\log_2N/\log_3N}}\ll (\log N)^{-1/2+\delta}
		\]
		for $p\in P_q$ and there are $\ll \log T/\log_3 T$ prime factors of any given $n$ in $\mc{M}_q$, the above is 
		\[
		\ll \sum_{n\in\mc{M}_q}f(n)^2\exp\Big(-\tfrac{\epsilon\delta}{12\mf{n}}\log T+o\Big(\frac{(\log T)^{1/2+\delta}}{\log_3 T}\Big)\Big)
		\]
		which is $o(\sum_{n\in\mc{M}_q}f(n)^2)$ when $\mf{n}\ll (\log T)^{1/2-\delta}$  and hence certainly $o(A_{N,q}\sum_{n\in\mb{N}}f(n)^2)$.
		\end{proof}

		\begin{proof}[Proof of Theorem \ref{main theorem}]
		
		Note that
		\begin{multline}\label{max pull}
		\bigg|\int_{1\leqs |t|\leqs T\log T}\int_{|u|\leqs T^{1/2}}\zeta_{\mb{K}}(1/2+i(t+u))K_{\mf{n}}(u)|R(t)|^2\Phi(t/T)dt du\bigg|
		\\
		\leqs \widehat{K_\mf{n}}(0)\max_{t\in[-T^{1/2},T\log T+T^{1/2}]}|\zeta_\mb{K}(\tfrac12+it)|\int_\mb{R}|R(t)|^2\Phi(t/T)dt.
		\end{multline}
		Thus it remains to extend the integrals on the left to $\mb{R}$. 
		
		Consider the region $|t|\leqs 1$ first. By the hybrid convexity bound $L(\tfrac12+iv,\chi)\ll (q(1+|v|))^{1/4}$ (or better, see \cite{HB})  
		we have $\zeta_\mb{K}(\tfrac12+iv)\ll (cq(1+|v|))^{\phi(q)/4}$. This leads to a contribution 
		\begin{align*}
		\ll  
		q^{\phi(q)}R(0)^2\int_{|u|\leqs T^{1/2}}(1+|u|)^{\phi(q)/4}K_\mf{n}(u)du
		\ll &
		q^{\phi(q)}T^{\eta}\sum_{n\in\mb{N}}f(n)^2 
		\end{align*}
		on taking 
		\[
		\mf{n}=2\phi(q),
		\] 		
		say. 
		We next extend the $u$ integral to $\mb{R}$. Using that $K_\mf{n}(u)\ll 1/u^{4\phi(q)}$ along with the convexity bound gives 
		\begin{align*}
		\ll &
		q^{\phi(q)}\int_{|t|\leqs T\log T}\int_{|u|\geqs T^{1/2}}
		(1+|t+u|)^{\phi(q)/4}K_\mf{n}(u)|R(t)|^2\Phi(t/T)dudt
		\\
		\ll &
		q^{\phi(q)}(T\log T)^{\phi(q)/4}T^{(1+\phi(q)/4-4\phi(q))/2}\int_\mb{R}|R(t)|^2\Phi(t/T)dt
		\\
		\ll & q^{\phi(q)}T^{-\phi(q)}\sum_{n\in\mb{N}}f(n)^2.
		\end{align*}
		Finally, extending the remaining $t$ integral to $\mb{R}$ gives an error
		\begin{align*}
		\ll &
		q^{\phi(q)}\int_{|t|\geqs T\log T}\bigg(\int_{\mb{R}}|u|^{\phi(q)/4}K_\mf{n}(u)du \bigg) \,\,|t|^{\phi(q)/4}|R(t)|^2\Phi(t/T)dt
		\\
		\ll &
		q^{\phi(q)}T^{\phi(q)/4+1+\eta}\Phi(\tfrac12\log T)\sum_{n\in\mb{N}}f(n)^2.
		\end{align*}
		
		For $q\ll (\log_2 T)^A$ these errors are all $o(T\sum_{n\in\mb{N}}f(n)^2)$ and we find 
		\begin{align*}
		&\int_{T^{\beta}\leqs |t|\leqs T \log T}\int_{|u|\leqs T^\beta/2}\zeta_{\mb{K}}(\tfrac12+i(t+u))K(u)|R(t)|^2\Phi(t/T)dt du\\ &=\int_{-\infty}^{\infty}\int_{-\infty}^{\infty}   
		\zeta_{\mb{K}}(\tfrac12+i(t+u))K(u)|R(t)|^2\Phi(t/T)dt du+o\Big(T\sum_{n\in\mb{N}}f(n)^2\Big). 
		\end{align*}
		Thus, on taking $q\ll(\log_2 T)^A$ the condition on $\mf{n}(=2\phi(q))$ in Lemma \ref{main lemma} is satisfied. Along with Lemma \ref{convolution formula} 		and 	\eqref{max pull} this gives 
		\begin{align*}
		 \max_{-T^{1/2}\leqs t\leqs T\log T+T^{1/2}}|\zeta_\mb{K}(\tfrac12+it)| 
		\gg &
		\exp\left(c_q \gamma(1+o(1)) \sqrt{\frac{\log N \log_3 N}{\log_2 N}}\right)
		\\
		&
		+
		o( \mathrm{Res}_{s=1}\zeta_\mb{K}(s))+o(1)
		\end{align*}
		since
		 \[
		 \int_{-\infty}^{\infty}|K_\mf{n}(-t-i/2)||R(t)|^2\phi(t/T)dt\ll T^{\eta+\epsilon}\sum_{n\in\mb{N}}f(n)^2 .
		 \]
		
		It is known from \cite{Ram} that
		\[
		\text{Res}_{s=1}\zeta_{\mb{K}}(s)\leqs \left(\frac{\log |d_{\mb {K}}|}{2(\phi(q)-1)}+\kappa\right)^{\phi(q)-1},
		\]
		where $\kappa= (5 -2 \log 6)/2 = 0.70824\cdots$.
		Since the discriminant satisfies $d_\mb{K}\ll q^{\phi(q)}$ (see Proposition 2.7 of \cite{Wash}), we have 
		\[
		\text{Res}_{s=1}\zeta_{\mb{K}}(s)\ll \exp(\log_3 T\log_5 T)
		\]  
		which is negligible. We now take $c_q=\sqrt{\phi(q)}$, $\eta<1$ and let $\gamma\to 1$ to complete the proof.
		Note that varying $T$ by a factor of a logarithm only affects the lower order terms in the exponential and so the result holds for $t$ in $[0,T]$, as stated. 
		\end{proof}

	\end{document}